\documentclass[12pt,a4paper]{amsart}

\usepackage{docmute}
\usepackage{a4wide}
\usepackage[utf8]{inputenc}
\usepackage{amsmath,amssymb}

\usepackage[
backend=biber,
style=numeric,
sorting=nyt
]{biblatex}
\usepackage{xcolor}
\usepackage{booktabs}
\usepackage{hyperref}
\hypersetup{
	colorlinks,
	citecolor=black,
	filecolor=black,
	linkcolor=blue,
	urlcolor=black
}
\usepackage{stmaryrd}
\usepackage{url}
\usepackage{longtable}
\usepackage[figuresright]{rotating}
\usepackage{amsthm}
\usepackage{bbold}
\usepackage{enumerate}
\usepackage{geometry}
\newtheorem{definition}{Definition}
\newtheorem{theorem}[definition]{Theorem}
\newtheorem{proposition}[definition]{Proposition}

\newtheorem{lemma}[definition]{Lemma}

\newtheorem{remark}[definition]{Remark}
\newtheorem{example}[definition]{Example}
\newtheorem{corollary}[definition]{Corollary}

\newtheorem*{claim*}{Claim}

\newcommand{\0}{\emptyset}
\newcommand{\mc}{\mathcal}

\newcommand{\RR}{\mathbb{R}}

\newcommand{\TT}{\mathbb{T}}

\newcommand{\ZZ}{\mathbb{Z}}

\newcommand{\foralmostall}{\forall^\infty}
\newcommand{\existinfty}{\exists^\infty}

\newcommand{\Ii}{\mc{I}}

\newcommand{\Kk}{\mc{K}}
\newcommand{\Mm}{\mc{M}}
\newcommand{\Nn}{\mc{N}}
\newcommand{\Ff}{\mc{F}}

\newcommand{\Pp}{\mc{P}}

\newcommand{\bez}{\backslash}
\newcommand{\se}{\subseteq}
\newcommand{\sen}{\subsetneq}
\newcommand{\es}{\supseteq}
\newcommand{\nes}{\supsetneq}

\newcommand{\rest}{\hspace{-0.25em}\upharpoonright\hspace{-0.25em}}

\newcommand{\baire}{\omega^\omega}

\newcommand{\cons}{^{\frown}}

\newcommand{\tn}[1]{\textnormal{#1}}
\newcommand{\ti}[1]{\textit{#1}}

\newcommand{\dom}{\textnormal{dom}}

\newcommand{\nwd}{\textnormal{nwd}}

\newcommand{\splitt}{\textnormal{split}}
\newcommand{\wsplit}{\w\textnormal{-split}}
\newcommand{\level}{\textnormal{level}}

\newcommand{\succe}{\textnormal{succ}}
\newcommand{\stem}{\textnormal{stem}}

\def\w{\omega}

\def\baire{\w^\w}

\title{On algebraic sums, trees and ideals in the Baire space}

\author[Ł. Mazurkiewicz]{Łukasz Mazurkiewicz}
\email{lukasz.mazurkiewicz@pwr.edu.pl}

\author[M. Michalski]{Marcin Michalski}
\email{marcin.k.michalski@pwr.edu.pl}

\author[R. Rałowski]{Robert Rałowski}
\email{robert.ralowski@pwr.edu.pl}

\author[Sz. Żeberski]{Szymon Żeberski}
\email{szymon.zeberski@pwr.edu.pl}

\thanks{The work has been partially financed by grant {\bf 8211204601, MPK: 9130730000} from the Faculty of Pure and Applied Mathematics, Wrocław University of Science and Technology.
	}

\address{Marcin Michalski, Robert Rałowski, Szymon Żeberski, Faculty of Pure and Applied Mathematics, Wrocław University of Science and Technology, 50-370 Wrocław, Poland}

\date{}

\begin{document}

\begin{abstract}
	We work in the Baire space $\ZZ^\w$ equipped with the coordinate-wise addition $+$. Consider a $\sigma-$ideal $\Ii$ and a family $\TT$ of some kind of perfect trees. We are interested in results of the form: for every $A\in \Ii$ and a tree $T\in\TT$ there exists $T'\in \TT, T'\se T$ such that $A+\underbrace{[T']+[T']+\dots +[T']}_{\text{n--times}}\in \Ii$ for each $n\in\w$.
	
  Explored tree types include perfect trees, uniformly perfect trees, Miller trees, Laver trees and $\w-$Silver trees. The latter kind of trees is an analogue of Silver trees from the Cantor space.
   
   Besides the standard $\sigma$-ideal $\Mm$ of meager sets, we also analyze $\Mm_-$ and fake null sets $\Nn$. The latter two are born out of the characterizations of their respective analogues in the Cantor space. The key ingredient in proofs were combinatorial characterizations of these ideals in the Baire space.
\end{abstract}

\maketitle

\section{Introduction and notation}
  We adopt the standard set-theoretical notation (see e.g. \cite{Jech}). Throughout the paper we usually refer to the space $\ZZ^\w$ as the Baire space for its algebraic structure, i.e. coordinate-wise addition $+$ defined by $(x+y)(n)=x(n)+y(n)$ for all $x,y\in \ZZ^\w$ and $n\in\w$.
  
  For $A,B\se \ZZ^\w$ we define the algebraic sum in the standard way
  \[
    A+B=\{a+b:\, a\in A,\; b\in B\}.
  \]
  We will use the same notation for translation via point $x\in\ZZ^\w$, i.e. $x+A=\{x\}+A$, and for addition in $\ZZ^n$ for any $n\in\w$. The context will be always clear and will not lead to confusion.
  
  If in a given context the algebraic structure is not important, we will refer to the canonical Baire space $\w^\w$.
  
  Occasionally we will highlight differences and similarities between the Baire space and the Cantor space. In such cases we treat the Cantor space as ${\ZZ_2}^\w$, also equipped with the coordinate wise addition $+$ (see \cite{MiRalZebAddCant}).
  
  Let us recall some notions regarding trees. Assume that $T\se\ZZ^{<\w}$ is a tree. Then
  \begin{itemize}
    \item $\succe_T(\sigma)=\{i\in\ZZ:\, \sigma\cons i\in T\}$;
    \item $\splitt(T)=\{\sigma\in T:\, |\succe_T(\sigma)|\geq2\}$;
    \item $\wsplit(T)=\{\sigma\in T:\, |\succe_T(\sigma)|=\w\}$.
  \end{itemize}

  \begin{definition}
    We call a tree $T\se \ZZ^{<\w}$
    \begin{itemize}
      \item perfect, if $(\forall\sigma\in T)(\exists\tau\in T)(\sigma\se\tau\land\tau\in\splitt(T))$;
      \item uniformly perfect, if for every $n\in\w$ either $\ZZ^n\cap T\subseteq\splitt(T)$ or $\ZZ^n\cap\splitt(T)=\0$;
      \item Miller, if $(\forall\sigma\in T)(\exists\tau\in T)(\sigma\se\tau\land\tau\in\wsplit(T))$;
      \item Laver, if $(\exists\sigma\in T)(\forall\tau\in T)(\tau\se\sigma\lor(\sigma\se\tau\land\tau\in\wsplit(T)))$;
      \item $\w-$Silver, if there are $A\in [\w]^\w$ and $x_T$ such that 
      \[
        T=\{\sigma\in \ZZ^{<\w}:\, (\forall n\in\dom(\sigma)\bez A)(\sigma(n)=x_T(n))\}.
      \]
    \end{itemize}
  \end{definition}
  Let us remark that the notion of $\w-$Silver seems to be a natural analogue of Silver trees living in the Cantor space that realizes the main feature of the Baire space ($\w-$splitting).
  
  We will denote the set of infinite branches of a tree $T\se \ZZ^{<\w}$ by $[T]$, i.e.
  \[
    [T]=\{x\in \ZZ^\w:\, (\forall n\in\w)(x\rest n\in T )\}.
  \] 
  
  The following remark pinpoints the reason why $\w-$Silver are easier to handle in comparison with other type of trees.
  \begin{remark}
      For every $\w-Silver$ tree $T$ we have $[T]+[T]=[T]+x_T$.
    \end{remark}

  Algebraic sums were mostly studied in the context of the real line with a standard addition. Results related to the ones presented in this paper were also helpful in \cite[Lemma 3]{Rec}, where the author proved that for every null set $A\se \RR$ and every perfect set $P\se \RR$ there exists a perfect set $P'\se P$ such that $A-P$ is null. Analogous result concerned with $+$ and $\sigma-$ideal of meager sets was proved in \cite[Theorem 11]{Scheepers}. Various similar results were also proved in \cite{ErdKuMa}, especially Lemma 9. Algebraic sums in a context of nonmeasurability were studied in \cite{NoScheeWeiss} and \cite{Ky}. Superfluously contradictory results appeared in \cite{MiRalZebNon}, where the authors obtained positive results regarding Miller and Laver trees localized via homeomorphism within irrational numbers in $\RR$.
  
  This paper can be considered a part II of \cite{MiRalZebAddCant}.
		
\section{Meager}

  Let us recall following characterization of meager sets in $2^\omega$ from \cite[Theorem 2.2.4]{BarJu}.
		\begin{lemma}\label{baza meagery}
			Let $F$ be a meager subset of $2^\omega$. There is $x_F\in 2^\omega$ and a partition $\{I_n: n\in\omega\}$ of $\omega$ into intervals such that
			\[
				F\se \{x\in 2^\omega: (\foralmostall n)(x\rest I_n \neq x_F\rest I_n)\}.
			\]
		\end{lemma}
		We define a family $\Mm_-\se P(\baire)$ in a similar fashion. $A\in\Mm_-$ if there is $x_A\in \omega^\omega$ and a partition $\{I_n: n\in\omega\}$ of $\omega$ into intervals such that
		\[
		  A\se \{x\in \w^\omega: (\foralmostall n)(x\rest I_n \neq x_A\rest I_n)\}.
		\]
		Exploiting the analogy to $\Mm$ in the Cantor space we will denote by $\nwd_-$ the ideal of sets generated by
		\[
		  \{x\in \w^\omega: (\forall n)(x\rest I_n \neq x_A\rest I_n)\}.
		\]
		Notice that $\Mm_-$ is a translation invariant $\sigma$-ideal with the basis of class $F_\sigma$. Also, $\Kk_{\sigma}\sen \Mm_-$ and $\Mm_-\se \Mm$. Moreover, the latter inclusion is proper, i.e. the characterization of $\Mm$ in the Cantor space à la Lemma \ref{baza meagery} fails for $\Mm$ in the Baire space.
	  \begin{theorem}
	    $\Mm\not\se\Mm_-$.
	  \end{theorem}
	  \begin{proof}
	    Let $f:\omega^{<\omega}\to\omega$ be a bijection and consider a tree
	    \[
	      T=\{\sigma\in\omega^{<\omega}:\; (\forall n<|\sigma|)(\sigma(n)\neq f(\sigma\rest n))\}.
	    \]
	    We will show that for any $y\in\w^\w$ and any partition $\{I_n: \;n\in\omega\}$ of $\omega$ into intervals there is $x\in [T]$ such that $x\rest I_n=y\rest I_n$ for infinitely many $n$. So, fix arbitrary $y\in\w^\w$ and a partition of $\omega$ into intervals $\{I_n:\; n\in\omega\}$. Let us start the induction on $n\in\omega$. At the step $0$ denote $I_1=[a_1, b_1]$ and consider a set
	    \begin{align*}
	      F_1=&\{\sigma\in T\cap \w^{a_1}:\; f(\sigma)=y(a_1) \lor f(\sigma\cons y(a_1))=y(a_1+1) \lor 
	      \\
	      &\lor f(\sigma\cons y(a_1)\cons y(a_1+1))=y(a_1+2) \lor\dots \lor f(\sigma\cons y\rest [a_1, b_1))=y(b_1)\}.
	    \end{align*}
	    It is finite (has at most $b_1-a_1+1$ elements), hence there is $\sigma_1\in T\cap \w^{a_1}\bez F_1$. Set $x_1=\sigma_1\cons y\rest I_1$. Clearly, $x_1\in T$. Let us assume that at the step $n+1$ we already have $x_{2n+1}\in T$ such that $x_{2n+1}\rest I_{2k+1}=y\rest I_{2k+1}$ for $k<n+1$. Denote $I_{2n+3}=[a_{2n+3}, b_{2n+3}]$ and consider a set
	    \begin{align*}
	      F_{2n+3}=&\{\sigma\in T\cap \w^{a_{2n+3}}:\; x_{2n+1}\se \sigma \land \big( f(\sigma)=y(a_{2n+3}) \lor 
	      \\
	      &\lor f(\sigma\cons y(a_{2n+3}))=y(a_{2n+3}+1) \lor 
	      \\
	      &\lor f(\sigma\cons y(a_{2n+3})\cons y(a_{2n+3}+1))=y(a_1+2) \lor\dots
	      \\
	      &\dots  \lor f(\sigma\cons y\rest [a_{2n+3}, b_{2n+3}))=y(b_{2n+3})\big)\}.
	    \end{align*}
	    It is finite, hence there is $\sigma_{n+1}\in T\cap \w^{a_{2n+3}}\bez F_{2n+3}$. Set $x_{2n+3}=\sigma\cons y\rest I_{2n+3}$. This finishes the inductive construction. Set $x=\bigcup_{n\in\omega}x_{2n+1}$. Clearly $x$ is the member of $[T]$ we are looking for.
	  \end{proof}
    
    We will rely on the following characterization of $\Mm$ in the Baire space.
    
    \begin{lemma}\label{charakteryzacja M}
      For every meager set $F\se \baire$ there exists $f: \omega^{<\omega}\to \omega^{<\omega}$ such that
      \[
        F\se \{x\in\baire: \, (\foralmostall \sigma\in \omega^{<\omega})(\sigma\cons f(\sigma)\not\se x)\}.
      \]
      Moreover, the set on the right is meager. 
    \end{lemma}
    \begin{proof}
      Let $F=\bigcup_{n\in\omega}F_n$, where $(F_n:\, n\in\omega)$ is an ascending sequence of nowhere dense sets. For each $n$ there exists $f_n: \omega^{<\omega}\to  \omega^{<\omega}$ such that for every $\sigma$ we have $[\sigma \cons f_n(\sigma)]\cap F_n=\0$. Notice that
      \[
        F_n\se \{x\in\baire:\, (\forall \sigma \in \omega^{<\omega})(\sigma\cons f_n(\sigma)\not\se x)\}.
      \]
      Furthermore we may assume that $f_{n}(\sigma)\se f_{n+1}(\sigma)$ for each $n\in\omega$. Let $\{\sigma_n:\, n\in\omega\}=\omega^{<\omega}$ and set $f(\sigma_n)=f_{n}(\sigma_n)$. The function $f$ is the one we are looking for.
      
      Indeed, let $x\in F$. Then there is $N\in\omega$ such that $x\in F_n$ for $n\ge N$. Then for $n\ge N$  $x\not\es {\sigma_n}\cons f_n(\sigma_n)={\sigma_n}\cons f(\sigma_n)$.
    \end{proof}
	
	  Now we are well prepared for the main results of this section. The following result nips in the bud any considerations concerning Laver trees.
	
	  \begin{proposition}
		There exists a set $A\in \mc{\Mm_{-}}$ such that $A+[T]=\ZZ^\omega$ for each Laver tree $T$. 
	  \end{proposition}
	  \begin{proof}
		  Define
		  \[
	  		A=\{x\in\ZZ^\omega: (\foralmostall n)(x(n)\neq 0)\}.
	  	\]
	  	Let $T$ be a Laver tree and let $\sigma_0=\stem(T)$. Let $z\in \ZZ^\omega$. We will find $x\in A$ and $y\in [T]$ satisfying $x+y=z$. Set $y\rest|\sigma_0|=\sigma_0$ and $x(n)=z(n)-y(n)$ for $n<|\sigma_0|$. Then set $\succe_T(y\rest n)\ni y(n)\neq z(n)$ and $x(n)=z(n)-y(n)$ for $n\geq |\sigma_0|$. 
	  \end{proof}		
	  \begin{remark}
		  In the above theorem it is sufficient for considered trees $L$ to satisfy $\sigma\in \splitt(L)$ for each $\sigma\es \stem(L)$.
	  \end{remark}
	  
	  Thanks to the characterization of $\Mm_-$ resembling the one of $\Mm$ in the Cantor space, we have the following two immediate observations.
    
    \begin{theorem}
      For every $F\in \Mm_-$ and every (uniformly) perfect tree $T\se Z^{<\w}$ there is a (uniformly) perfect tree $T'\se T$ such that
      \[
        F+\underbrace{[T']+[T']+\dots +[T']}_{n\ti{--times}}\in \Mm_-.
      \]
    \end{theorem}
    \begin{proof}
      Almost identical to the proof of \cite[Theorem 6]{MiRalZebAddCant}.
    \end{proof}
    
    \begin{theorem}
      For every $F\in \Mm_-$ and every $\w-$Silver tree $T\se Z^{<\w}$ there is a $\w-$Silver tree $T'\se T$ such that
      \[
        F+\underbrace{[T']+[T']+\dots +[T']}_{n\ti{--times}}\in \Mm_-.
      \]
    \end{theorem}
    \begin{proof}
      Almost identical to the proof of \cite[Theorem 5]{MiRalZebAddCant}.
    \end{proof}
    
    The case of perfect trees and meager sets is more nuanced.
    
    \begin{theorem}\label{Meager 1 Sacks}
      For every $F\in \Mm$ and every (uniformly) perfect tree $T\se \ZZ^{<\w}$ there is a (uniformly) perfect tree $T'\se T$ such that
      \[
        F+[T']\in \Mm.
      \]
    \end{theorem}
    \begin{proof}
      Let $T$ be a perfect tree (the proof for a uniformly perfect tree is almost identical). Let $\{\sigma_n:\, n\in\omega\}=\ZZ^{<\w}$ and assume that if $\sigma_n \se \sigma_m$, then $n\leq m$. Let $F$ be meager with $f$ as in Lemma \ref{charakteryzacja M}, i.e.
      \[
        F\se\{x\in\ZZ^\w:\, (\foralmostall n \in \omega)({\sigma_n}\cons f(\sigma_n)\not\se x)\}.
      \]
      For every $n\in \w$ let the enumeration $\{\rho^n_k:\, k<2^n\}=2^n$ be in lexicographical order, i.e. $\rho^{n+1}_{2k}={\rho^n_k}\cons 0$ and $\rho^{n+1}_{2k+1}={\rho^n_k}\cons 1$.
      We will construct inductively $(\tau_{\rho}:\, \rho\in 2^{<\w})$, $(\sigma^n_{k}:\, n\in \omega, k<2^n)$,  such that for each $n\in\omega$
      \begin{enumerate}[(i)]
        \item for each $\rho\in 2^n$ $\tau_\rho\in \splitt(T)$ and for $i=0,1$ $\tau_\rho\se{\tau}_{\rho\cons i}'\se\tau_{\rho\cons i}$ and ${\tau}'_{\rho\cons 0}\perp {\tau}'_{\rho\cons 1}$;
        \item ${\sigma_n}\cons \underbrace{00\dots 0}_{n-times}=\sigma^{n}_{-1}\se {\sigma^{n}_k}\se \sigma^n_{k+1}$ for $k<2^{n}-1$;
        \item $|\tau'_{\rho^n_k}| = |(\sigma^n_{k-1} - \tau'_{\rho^n_k}\rest |\sigma^n_{k-1}|)\cons f(\sigma^n_{k-1} - \tau'_{\rho^n_k}\rest |\sigma^n_{k-1}|)|$;
        \item if $(\forall m\ge n)({\sigma_m}\cons f(\sigma_m)\not\se x)$, then $\sigma^n_k \not\se x+\tau'_{\rho^n_k}$.
      \end{enumerate}
      
      Let $|\tau'_\0|=|f(\0)|$ and let $\tau_{\0}$ be the shortest splitting extension of $\tau'_\0$ from $T$. Set $\sigma^0_0=f(\0)+\tau'_\0$.
      
      Let us consider the step $n+1>0$. Set $T\ni\tau'_{{\rho^n_0} \cons 0}\es {\tau_{\rho^n_0}}\cons i^0_0$ long enough so that the following hold  
      \[
        |\tau'_{{\rho^n_0} \cons 0}| = |({\sigma^{n+1}_{-1}-\tau'_{{\rho^n_0} \cons 0}\rest |\sigma^{n+1}_{-1}|})\cons f(\sigma^{n+1}_{-1}-\tau'_{{\rho^n_0} \cons 0}\rest |\sigma^{n+1}_{-1}|)|.
      \]
      Denote $\sigma^{n+1}_{0}=({\sigma^{n+1}_{-1}-\tau'_{{\rho^n_0} \cons 0}\rest |\sigma^{n+1}_{-1}|})\cons f(\sigma^{n+1}_{-1}-\tau'_{{\rho^n_0} \cons 0}\rest |\sigma^{n+1}_{-1})+\tau'_{{\rho^n_0}\cons 0}$. In a similar fashion, set $T\ni\tau'_{{\rho^n_0} \cons 1}\es {\tau_{\rho^n_0}}\cons i^0_1$, $i^0_1\neq i^0_0$, such that
      \[
        |\tau'_{{\rho^n_0} \cons 1}| = |({\sigma^{n+1}_0-\tau'_{{\rho^n_0} \cons 1}\rest |\sigma^{n+1}_0|})\cons f(\sigma^{n+1}_0-\tau'_{{\rho^n_0} \cons 1}\rest |\sigma^{n+1}_0|)|
      \]
      and denote $\sigma^{n+1}_{1}=({\sigma^{n+1}_0-\tau'_{{\rho^n_0} \cons 1}\rest |\sigma^{n+1}_0|})\cons f(\sigma^{n+1}_0-\tau'_{{\rho^n_0} \cons 1}\rest |\sigma^{n+1}_0)+\tau'_{{\rho^n_0}\cons 1}$. Notice that $\sigma_{n+1}\se \sigma^{n+1}_{-1}\se \sigma^{n+1}_{0}\se \sigma^{n+1}_{1}$. At the step $k<2^{n}$ assume that $\sigma^{n+1}_{2j}$ and $\sigma^{n+1}_{2j+1}$ for $j<k$ are defined. Set $\tau'_{{\rho^n_k} \cons 0}\es {\tau_{\rho^n_k}}\cons i^k_0$ and $\tau'_{{\rho^n_k} \cons 1}\es {\tau_{\rho^n_k}}\cons i^k_1$, $i^k_0\neq i^k_1$, from $T$ such that
      \begin{align*}
        |\tau'_{{\rho^n_k} \cons 0}| = |({\sigma^{n+1}_{2k-1}-\tau'_{{\rho^n_k} \cons 0}\rest |\sigma^{n+1}_{2k-1}|})\cons f(\sigma^{n+1}_{2k-1}-\tau'_{{\rho^n_k} \cons 0}\rest |\sigma^{n+1}_{2k-1}|)|,
        \\
        |\tau'_{{\rho^n_k} \cons 1}| = |({\sigma^{n+1}_{2k}-\tau'_{{\rho^n_k} \cons 1}\rest |\sigma^{n+1}_{2k}|})\cons f(\sigma^{n+1}_{2k}-\tau'_{{\rho^n_k} \cons 1}\rest |\sigma^{n+1}_{2k}|)|,
      \end{align*}
      where, predictably,
      \[
        \sigma^{n+1}_{2k}=({\sigma^{n+1}_{2k-1}-\tau'_{{\rho^n_{k}} \cons 0}\rest |\sigma^{n+1}_{2k-1}|})\cons f(\sigma^{n+1}_{2k-1}-\tau'_{{\rho^n_{k}} \cons 0}\rest |\sigma^{n+1}_{2k-1})+\tau'_{{\rho^n_{k}}\cons 0}.
      \]      
      Also, set
      \[
        \sigma^{n+1}_{2k+1}=({\sigma^{n+1}_{2k}-\tau'_{{\rho^n_{k}} \cons 1}\rest |\sigma^{n+1}_{2k}|})\cons f(\sigma^{n+1}_{2k}-\tau'_{{\rho^n_{k}} \cons 1}\rest |\sigma^{n+1}_{2k})+\tau'_{{\rho^n_{k}}\cons 1}.
      \]
      
      Observe that $\sigma^{n+1}_{2k-1}\se \sigma^{n+1}_{2k}\se\sigma^{n+1}_{2k+1}$. Finally set $\tau_{\rho\cons i}$, $\rho\in 2^n, i=0,1,$ to be the shortest extensions of $\tau'_{\rho\cons i}$ to splitting nodes of $T$. The construction is complete.
      
      Clearly (ii) is the case. Conditions (i) and (iii) are also satisfied.
      
      To see (iv) let $x\in \ZZ^\w$ be such that ${\sigma_m}\cons f(\sigma_m)\not\se x$ for $m\ge n+1$ for some $n$. Let $k<2^{n+1}$. Then $\sigma^{n+1}_k=(\sigma^{n+1}_{k-1}-\tau'_{\rho^{n+1}_k}\rest |\sigma^{n+1}_{k-1}|)\cons f(\sigma^{n+1}_{k-1}-\tau'_{\rho^{n+1}_k}\rest |\sigma^{n+1}_{k-1}|)+\tau'_{\rho^{n+1}_k}$ and clearly $|(\sigma^{n+1}_{k-1}-\tau'_{\rho^{n+1}_k}\rest |\sigma^{n+1}_{k-1}|)|> n+1$. Hence $(\sigma^{n+1}_{k-1}-\tau'_{\rho^{n+1}_k}\rest |\sigma^{n+1}_{k-1}|)\cons f(\sigma^{n+1}_{k-1}-\tau'_{\rho^{n+1}_k}\rest |\sigma^{n+1}_{k-1}|)\not\se x$, so $(\sigma^{n+1}_{k-1}-\tau'_{\rho^{n+1}_k}\rest |\sigma^{n+1}_{k-1}|)\cons f(\sigma^{n+1}_{k-1}-\tau'_{\rho^{n+1}_k}\rest |\sigma^{n+1}_{k-1}|)+\tau'_{\rho^{n+1}_{k}}\not\se x+\tau'_{\rho^{n+1}_{k}}$.

      For every $n\in\omega$ let $h(\sigma_n)$ be such that $\sigma^{n}_{2^n-1}={\sigma_n}\cons h(\sigma_n)$. The function $h$ is well defined thanks to (ii).
      
      Set 
      \begin{align*}
        T'&=\{\tau\in T:\, (\exists \rho\in 2^{<\w})(\tau\se \tau_{\rho})\},
        \\
        H&=\{x\in \ZZ^{\w}:\, (\foralmostall n \in \omega)({\sigma_n}\cons h(\sigma_n)\not\se x)\}.
      \end{align*}
      We will show that $F+[T']\se H$. Let $x\in F$ and $t\in [T']$. Let $N\in \w$ be such that ${\sigma_n}\cons f(\sigma_n)\not\se x$ for $n\ge N$. Fix such $n$. Let $k<2^n$ be such that $\tau'_{\rho^n_k}\se t$. By (iv) $\sigma^n_k\not\se x+\tau'_{\rho^n_k}$. Notice that $\sigma^n_k\se \sigma^n_{2^n-1}={\sigma_n}\cons h(\sigma_n)$. Hence,  ${\sigma_n}\cons h(\sigma_n)\not\se x+t$.
    \end{proof}
    
    \begin{theorem}\label{Sacks fuzja}
      For every $F\in \Mm$ and every (uniformly) perfect tree $T\se \ZZ^{<\w}$ there is a (uniformly) perfect tree $T'\se T$ such that for each $n$
      \[
        F+\underbrace{[T']+[T']+\dots +[T']}_{n-times}\in \Mm.
      \]
    \end{theorem}
    \begin{proof}
      Let
      \begin{align*}
        \level(T,0)=&\{\stem(T)\},
        \\
        \level(T,n+1)=&\{\tau\in\splitt(T):\, (\exists \sigma\in \level(T,n))(\sigma\sen \tau \land
        \\
        &(\forall \eta\in T)(\sigma\sen \eta\sen \tau \to \eta\notin\splitt(T)))\}.
      \end{align*}
      For $n\in\w$ let $P\preceq_n Q$ if $P\se Q$ and $\level(P,n)=\level(Q,n)$.
      
      Using repeatedly Theorem \ref{Meager 1 Sacks} we may find for any meager set $F$ and (uniformly) perfect tree $T$ a fusion sequence of trees $(T_n:\, n\in\omega)$, i.e. 
      \[
        T_0=T,\quad F_0=F,\quad F_{n+1}=F_n+[T_{n+1}],\quad T_{n+1}\preceq_{n} T_n.
      \]
      
      Then $[T']=\bigcap_{n\in\omega}[T_n]$ is a body of the desired (uniformly) perfect tree.
    \end{proof}
    
    The answer in the case of Miller trees and meager sets is far from positive.
    
    \begin{example}
      There is an $\nwd_-$ set $F$  and a Miller tree $T$ such that for any Miller tree $T'\se T$
      \[
        F+[T']\notin\Mm.
      \]
    \end{example}
    \begin{proof}
      Let
      \begin{align*}
        F=\{x\in \ZZ^\w:\, (\forall n)(x(n)\neq 0)\},
      \end{align*}
      Fix a bijection $\alpha: \, \ZZ^{<\w}\to \ZZ$ and let $\hat{\alpha}:\, \ZZ^{<\w}\to \ZZ^{<\w}$ be given by
      \begin{align*}
        \hat{\alpha}(\0)&=\0,
        \\
        \hat{\alpha}(\sigma\cons i)&=\hat{\alpha}(\sigma)\cons \alpha(\sigma\cons i).
      \end{align*}
      Set $T= \rm{rng}(\hat{\alpha})$. Clearly $T$ is a Miller tree. Let $T'\se T$ be a Miller tree.    
      
      We will show that for every $H\in\Mm$ there are $x\in F, t\in [T']$ such that $x+t\notin H$.
      
      Fix $h:\ZZ^{<\w} \to \ZZ^{<\w}$ and set
      \[
        H=\{y:\, (\foralmostall \sigma)(y\not\es \sigma\cons h(\sigma))\}.
      \]
      We will find $x\in F$ and $t\in [T']$ such that $x+t\es \sigma\cons h(\sigma)$ for infinitely many $\sigma$. For this purpose let us construct $(\tau_n:\, n\in\w)$, $(\tau'_n:\, n\in\w)$, $(\xi_n:\, n\in\w), (\sigma_n:\, n\in\w)$ such that
      \begin{enumerate}[(i)]
        \item $|\sigma_{n+1}|>|\sigma_n|$;
        \item $\tau_n\se \tau'_n\se \tau_{n+1}$, $\tau'_n\in \wsplit(T')$;
        \item $\xi_n\se \xi_{n+1}$, $|\xi_n|=|\tau_n|$, $\xi_n(k)\neq 0$ for $k\in\dom(\xi_n)$;
        \item $\xi_n+\tau_n\es{\sigma_n}\cons h(\sigma_n)$.
      \end{enumerate}
      
      Let $|\sigma_0|=\stem(T')$ and $\sigma_0(k)\neq \stem(T')(k)$ for $k<|\stem(T')|$. Let $\tau_0\in T'$ be such that $\tau_0(k)\neq ({\sigma_0}\cons h(\sigma_0))(k)$ for $k<|{\sigma_0}\cons h(\sigma_0)|$. Set $\xi_0={\sigma_0}\cons h(\sigma_0)-\tau_0$. Let $\tau'_0\es \tau_0$ with $\tau'_0\in\wsplit(T')$.
      
      At the step $n+1$ let $\sigma_{n+1}\es {\sigma_{n}}\cons h(\sigma_n)$, $|\sigma_{n+1}|=\tau'_{n}$ and 
      \[
        (\forall k\in \dom(\sigma_{n+1})\bez \dom({\sigma_n}\cons h(\sigma_n)))(\sigma_{n+1}(k)\neq\tau_n'(k)).
      \]
      Set $\tau_{n+1}\es \tau'_{n}$ from $T'$ such that
      \[
      (\forall k\in\dom({\sigma_{n+1}}\cons h(\sigma_{n+1})\bez \dom(\tau'_n))(\tau_{n+1}(k)\neq ({\sigma_{n+1}}\cons h(\sigma_{n+1}))(k)).
      \]
      Finally set $\tau_{n+1}'\in\wsplit(T'), \tau'_{n+1}\es \tau_{n+1}$. The construction is complete.
      
      Let $x=\bigcup_{n\in\w}\xi_n$ and $t=\bigcup_{n\in\w}\tau_n$. Clearly $x\in F$ and $t\in [T']$. Furthermore by (iv) $x+t\es {\sigma_n}\cons h(\sigma_n)$, hence $x+t\notin H$.
    \end{proof}
    
    Replacing the Miller tree with $\w-$Silver tree does not help much.
    
    \begin{example}
      There is a nowhere dense set $F$ such that for each $\w-$Silver tree $T$ we have $F+[T]\notin \Mm$.
    \end{example}
    \begin{proof}
      Without loss of generality we may assume that $x_T=(0,0,\dots)$, i.e. $[T]=\{x\in \ZZ^\w:\, (\forall n\notin A)(x(n)=0)\}$. Let $F=\{x\in \ZZ^{\w}:\, (\forall n\in\w)(x\not\es {\sigma_n}\cons \underbrace{0\dots 0}_{n-times})\}$, where $(\sigma_n:\, n\in\w)$ is an enumeration of $\ZZ^{<\w}$ such that $n<m$ for $\sigma_n\se \sigma_m$. Let $H$ be any meager set associated with a function $h$. We will construct $(\tau_n:\, n\in \w)\in T^\w$ and $(\rho_n:\,n\in\w)\in (\ZZ^{<\w})^\w$ such that for all $n\in\w$
      \begin{enumerate}[(i)]
        \item $\tau_n\sen \tau_{n+1}, \rho_n\sen \rho_{n+1}$;
        \item $(\forall k\in\w)(\rho_n\not\es {\sigma_k}\cons \underbrace{0\dots 0}_{k-times})$;
        \item $\tau_n(k)=0$ for $k\in \dom(\tau_n)\bez A$;
        \item ${(\tau_n+\rho_n)}\cons h(\tau_n+\rho_n)\se \tau_{n+1}+\rho_{n+1}$.
      \end{enumerate}
      Let $\rho_0'=\underbrace{1\dots 1}_{\min A}$, $\tau_0'=\underbrace{0\dots 0}_{\min A}$ and $\rho_0={\rho_0'}\cons l_0=\sigma_m$, where $m>|h({\rho_0'}\cons 1)|$, $\tau_0={\tau_0'}\cons (1-l_0)$. Let $\rho_1'={\rho_0}\cons h({\rho_0'}\cons 1)\cons 1\dots 1$ such that $|\rho_1'|\in A$ and $\tau_1'={\tau_0'}\cons 0\dots 0$ with $|\tau_1'|=|\rho_1'|$. Set $\rho_1={\rho_1'}\cons l_1=\sigma_m$, where $m>|h((\tau_1'+\rho_1')\cons 1)|$, and $\tau_1={\tau_1'}\cons (1-l_1)$.
      
      Let us execute the step $n+1$. Let $\rho_{n+1}'={\rho_n} \cons h(\rho_n+\tau_n)\cons 1\dots 1$ and $\tau_{n+1}'={\tau_n}\cons 0\dots \cons 0$ such that $|\rho_{n+1}'|=|\tau_{n+1}'|\in A$. Set $\rho_{n+1}={\rho_{n+1}'}\cons l_{n+1}=\sigma_m$, where $m>|h((\tau_{n+1}'+\rho_{n+1}')\cons 1)|$, and $\tau_{n+1}={\tau_{n+1}'}\cons (1-l_{n+1})$. The construction is complete.
      
      Set $t=\bigcup_{n\in\w}\tau_n$ and $x=\bigcup_{n\in\w}\rho_n$. By (ii) $x\in F$, by (iii) $t\in [T]$, and by (iv) $x+t\notin H$.
    \end{proof}
    
  \section{Fake null}
    
    It is known that there is no translation invariant regular measure on $\ZZ^\w$, as the latter is not locally compact. It does not mean however, that one cannot define a reasonable translation invariant $\sigma-$ideal resembling null sets in the Cantor space.
    \begin{definition}
      We will say that a set $A$ is fake null, denote by $A\in\Nn$, if
      \[
        (\forall \varepsilon>0)(\exists (\sigma_n:\, n\in\w))\left(\sum_{n\in\w} \frac{1}{2^{|\sigma_n|}}<\varepsilon\,\&\,A\se \bigcup_{n\in\w}[\sigma_n]\right).
      \]
    \end{definition}
    
    Clearly it is a translation invariant $\sigma-$ideal. Moreover it is orthogonal to $\Mm$, i.e. there is a comeager set $G\in\Nn$. There is a compact set which is not fake null, e.g. body of any full binary tree. Also, the characterization from \cite[Lemma 2.5.1]{BarJu} works, namely
    \begin{lemma}\label{characterization of fake null}
      Let $F\in \Nn$. Then there is a sequence $(S_n:\, n\in\w)$, $S_n\se \ZZ^n$ for each $n\in\w$, such that $\sum_{n\in\w}\frac{|S_n|}{2^n}<\infty$ and
      \[
        F\se \{x\in \ZZ^\w:\, (\existinfty n\in\w)(x\rest n\in S_n)\}.
      \] 
      Conversely, if $(S_n:\, n\in\omega)$, $S_n\se \ZZ^n$, satisfy $\sum_{n\in\w}\frac{|S_n|}{2^n}<\infty$, then
      \[
        \{x\in\ZZ^\w:\, (\existinfty n\in\w)(x\rest n\in S_n)\}\in\Nn.
      \] 
    \end{lemma}
    \begin{proof}
      Let $F\se\bigcap_{n\in\w}\bigcup_{k\in\w}[\sigma^n_k]$, where $\bigcup_{k\in\w}[\sigma^n_k]\es F$ and $\sum_{k\in\w}\frac{1}{2^{|\sigma^n_k|}}<\frac{1}{2^n}$. For every $n$ set
      \[
        S_n=\{\sigma\in \ZZ^n:\, \sigma=\sigma^n_k \tn{ for some } n,k\in\w\}.
      \]
      
      Let
      \[
        F'=\{x\in \ZZ^\w:\, (\existinfty n\in\w)(x\rest n\in S_n)\}.
      \]
      
      See that $|S_n|<\sum_{k\le n}2^{n-k}$, hence $\sum_{n\in\w}\frac{|S_n|}{2^n}<\infty$.
      
      Furthermore $F\se F'$. To prove this let $x\in F$. Then there is $k_0\in\w$ such that $x\in [\sigma^0_{k_0}]$, hence $x\rest |\sigma^0_{k_0}|\in S_{|\sigma^0_{k_0}|}$. Assume that we have $0=N_0<N_1<\dots < N_n$ such that $x\rest |\sigma^{N_j}_{k_j}|\in S_{|\sigma^{N_j}_{k_j}|}$ and $|\sigma^{N_j}_{k_j}|<|\sigma^{N_{j+1}}_{k_{j+1}}|$ for $0\le j<n$ for some $n$. Set $N_{n+1}>|\sigma^{N_{n}}_{k_{n}}|$. Then there is $k_{n+1}$ such that $x\in [\sigma^{N_{n+1}}_{k_{n+1}}]$ which implies $x\rest |\sigma^{N_{n+1}}_{k_{n+1}}|\in S_{|\sigma^{N_{n+1}}_{k_{n+1}}|}$. Since $\sum_{k\in\w}\frac{1}{2^{|\sigma^{N_{n+1}}_k|}}<\frac{1}{2^{N_{n+1}}}$ it is the case that $|\sigma^{N_{n+1}}_{k_{n+1}}|>N_{n+1}>|\sigma^{N_{n}}_{k_{n}}|$. The induction is complete and it clearly results in $x\in F'$.
      
    
    To prove the second part, let $(S_n:\, n\in\omega)$, $S_n\se \ZZ^n$, satisfy $\sum_{n\in\w}\frac{|S_n|}{2^n}<\infty$. Then
      \[
        \{x\in\ZZ^\w:\, (\existinfty n\in\w)(x\rest n\in S_n)\}=\bigcap_{n\in\w}\bigcup_{k>n}\bigcup_{\sigma\in S_k}[\sigma].
      \]
      The set $\bigcup_{k>n}\bigcup_{\sigma\in S_k}[\sigma]$ is covered by basic clopen sets for which
      \[
        \sum_{k>n}\sum_{\sigma\in S_k}\frac{1}{2^{|\sigma|}}=\sum_{k>n}\sum_{\sigma\in S_k}\frac{1}{2^{k}}=\sum_{k>n}\frac{|S_k|}{2^{k}}\xrightarrow{n\to\infty} 0.
      \]
    \end{proof}
    
    We will use this characterization to prove the following results.
    
    \begin{theorem}
      For every $F\in \Nn$ and every (uniformly) perfect tree $T\se Z^{<\w}$ there is a (uniformly) perfect tree $T'\se T$ such that for each $n$
      \[
        F+\underbrace{[T']+[T']+\dots + [T']}_{n-times}\in \Nn.
      \]
    \end{theorem}
    \begin{proof}
      Let $F\in \Nn$ and let $T\se \ZZ^{<\w}$ be a perfect tree (the proof for uniformly perfect trees is identical). Let $S_n\se \ZZ^n$ for $n\in\omega$ such that $\sum_{n\in\w}\frac{|S_n|}{2^n}<\infty$ and $F\se \{x\in \ZZ^\w:\, (\existinfty n\in\w)(x\rest n \in S_n)\}$. Let $(k_n:\, n\in\w)$ be a non-decreasing sequence of naturals  such that $\sum_{n\in\w}2^{(k_n)^3}\frac{|S_n|}{2^n}<\infty$ and $\lim_{n\to\infty}k_n=\infty$ (see \cite[Lemma 12]{MiRalZebAddCant}). Let $m_0=0$ and $m_{n+1}=\min\{m:\, m>m_n\;\& \;k_m>k_{m_n}\}$. We may find $(\tau_\sigma: \, \sigma\in 2^{<\w})\in T^{2^{<\w}}$ such that
      \begin{enumerate}[(i)]
        \item if $\sigma\in 2^n$ then $|\tau_\sigma|\ge m_n$;
        \item $\tau_{\sigma}\se \tau_{\sigma'}$ if $\sigma\se \sigma'$;
        \item $\tau_\sigma\in\splitt(T)$ and $\tau_{\sigma\cons 0}  \perp \tau_{\sigma\cons 1}$.
      \end{enumerate}
      Set
      \begin{align*}
        T'&=\{\tau\in T: (\exists \sigma\in 2^{<\w})(\tau \se \tau_\sigma)\},
        \\
        S_n'&=S_n+\bigcup_{j=1}^{k_n}(\underbrace{(T'\cap 2^n)+(T'\cap 2^n)+\dots+(T'\cap 2^n)}_{j-times}),
        \\
        F'&=\{x\in \ZZ^\w:\, (\existinfty n\in\w)(x\rest n\in S_n')\}.
      \end{align*}
      Notice that $|S_n'|\le |S_n|\cdot \sum_{j=1}^{k_n}\prod_{i=1}^j2^{k_n}\le |S_n|\cdot 2^{(k_n)^3}$, hence $\sum_{n\in\omega}\frac{|S_n'|}{2^n}<\infty$. Clearly, for every $n\in\w$
      \[
        F+\underbrace{[T']+[T']+\dots + [T']}_{n-times}\se F'\in \Nn.
      \]
    \end{proof}
    
    The last results will be concerned with Miller trees.
    
    \begin{proposition}
      Every Miller tree $T$ contains a Miller tree $T'$ such that $[T']\in\Nn$.
    \end{proposition}
    
    \begin{proof}
      Let $T\se \ZZ^{<\w}$ be a Miller tree. Let $\w^{\uparrow <\w}$ denote the set of strictly increasing finite sequences. We will construct $\{\tau_{\sigma}:\, \sigma\in \w^{\uparrow <\w}\}\se T^{<\w}$, $(n_k:\, k\in\w)$ such that for each $\sigma\in \w^{\uparrow<\w}$
      \begin{enumerate}[(i)]
        \item $\tau_\sigma\in \wsplit(T)$;
        \item $\tau_\sigma\sen\tau_{\sigma\cons i}$ for $i>\max\sigma$;
        \item $\tau_{\sigma\cons i}(|\tau_\sigma|) \ne \tau_{\sigma\cons j}(|\tau_\sigma|)$ for $i\ne j$ and $i,j>\max\sigma$;
        \item $|\tau_\sigma|\ge 2k$ for $\max\sigma=k$.
      \end{enumerate}
      
      Let $\tau_\0=\stem(T)$. Assume we already have sequences $\tau_\sigma$ for $\sigma\in \w^{\uparrow <\w}$, $\max \sigma<k$ at the step $k$. For every $\sigma\in \w^{\uparrow <\w}$, $\max \sigma<k$ pick, $\tau_{\sigma\cons k}\nes \tau_\sigma$ satisfying (i) - (iv). Set
      \begin{align*}
        T'&=\{\tau\in \ZZ^{<\w}:\, (\exists \sigma\in \w^{\uparrow<\w})(\tau\se\tau_\sigma)\},
        \\
        S_{2k}&=\{\tau_\sigma\rest 2k:\, \sigma\in \w^{\uparrow<\w}, \max\sigma=k\}
      \end{align*}
      and $S_{2k+1}=\0$ for $k\in\w$. Clearly, $T'$ is a Miller tree contained in $T$. Also see that
      \[
        |S_{2k}| \le |\{\sigma\in \w^{\uparrow<\w}:\, \max\sigma=k\}| \le 2^k,
      \]
      hence $\sum_{n}\frac{|S_n|}{2^n}<\infty$. Moreover $[T']\se \{x\in \ZZ^\w:\, (\existinfty n)(x\rest n\in S_n)\}$. Indeed, if $t\in[T']$, then $t=\bigcup_{m\in\w}\tau_{y\;\rest\; m}$ for some $y\in \w^{\uparrow \w}$. Then $t\rest 2k\in S_{2k}$ for $k\in \tn{rng}(y)$.
    \end{proof}
    
    \begin{theorem}
      $[T_1]+[T_2]\notin\Nn$ for any Miller trees $T_1$, $T_2$.
    \end{theorem}
    \begin{proof}
      Let $T_1$ and $T_2$ be Miller trees and suppose that $[T_1]+[T_2]\in \Nn$. Let $(S_n:\, n\in \w)$ witness this fact as in Lemma \ref{characterization of fake null}. We will construct sequences $(\sigma_n:\, n\in\w)\in {T_1}^\w$, $(\tau_n: \, n\in\w)\in {T_2}^\w$ such that
      \begin{enumerate}[(i)]
        \item $\sigma_{n}\sen \sigma_{n+1}$ and $\tau_{n}\sen\tau_{n+1}$ for $n\in\w$;
        \item $|\sigma_n|=|\tau_n|$ for all $n\in\w$;
        \item $(\sigma_n+\tau_n)\rest k \notin S_k$ for $k\in(|\sigma_{n-1}|, |\sigma_{n}|], n>0$.
      \end{enumerate}
      Without loss of generality assume $|\stem(T_1)|\le |\stem(T_2)|$. Let $\sigma_0=\stem(T_1)$ and $\tau_0=\stem(T_2)\rest |\sigma_0|$. Let $\tau_1\es \tau_0$, $\tau_1\in T_2$ such that $\tau_1\in \wsplit(T)$ and $\sigma_1\in T_1$, $|\sigma_1|=|\tau_1|$, such that
      \[
        \sigma_1(|\tau_0|)\in \succe_{T_1}(\sigma_0)\bez \{\eta(|\tau_0|)-\tau_1(|\tau_0|):\, \eta\in S_k\, , \, k\in (|\tau_0|, |\tau_1|]\}.
      \]
      
    Assume that in an even step $2n$ $\tau_{2n-1}\in \wsplit(T_2)$. Pick $\sigma_{2n}\nes \sigma_{2n-1}$ such that $\sigma_{2n}\in \wsplit(T_1)$ and $\tau_{2n}\es \tau_{2n-1}$ such that $|\tau_{2n}|=|\sigma_{2n}|$ and
    \[
      \tau_{2n}(|\sigma_{2n-1}|)\in \succe_{T_2}(\tau_{2n-1})\bez \{\eta(|\sigma_{2n-1}|)-\sigma_{2n}(|\sigma_{2n-1}|):\, \eta\in S_k\, , \, k\in (|\sigma_{2n-1}|, |\sigma_{2n}|]\}.
    \]
	  We proceed similarly at an odd step, just swap the role of $\sigma$ and $\tau$.
	  
	  Set $s=\bigcup_{n\in\w}\sigma_n$ and $t=\bigcup_{n\in\w}\tau_n$.Clearly $s\in [T_1], t\in [T_2]$ and by (iii) $(s+t)\rest n \notin S_n$ for $n>0$.
	  \end{proof}
	  
	  \begin{corollary}
	    There exists a fake null set $F$ such that for any Miller tree $T$ $F+[T]\notin\Nn$.
	  \end{corollary}
	  \begin{corollary}
	    There exists a fake null set $F$ such that for any $\w-$Silver tree $T$ $F+[T]\notin\Nn$.
	  \end{corollary}
	  
	  \printbibliography

\end{document}